\documentclass[11pt]{article}
\usepackage{amsfonts, amssymb, amsmath, amsthm, xcolor, graphicx, cite}

\begin{document}

\theoremstyle{plain}
\newtheorem{thm}{Theorem}[section]
\newtheorem{cor}[thm]{Corollary}
\newtheorem{con}[thm]{Conjecture}
\newtheorem{cla}[thm]{Claim}
\newtheorem{lm}[thm]{Lemma}
\newtheorem{prop}[thm]{Proposition}
\newtheorem{example}[thm]{Example}

\theoremstyle{definition}
\newtheorem{dfn}[thm]{Definition}
\newtheorem{alg}[thm]{Algorithm}
\newtheorem{prob}[thm]{Problem}
\newtheorem{rem}[thm]{Remark}

\renewcommand{\baselinestretch}{1.1}

\title{\bf An analogue of the Erd{\H o}s Matching Conjecture for permutations with fixed number of cycles}
\author{
Cheng Yeaw Ku
\thanks{ Division of Mathematical Sciences, School of Physical and Mathematical Sciences,
Nanyang Technological University, 21 Nanyang link, Singapore 637371, Singapore. E-mail: cyku@ntu.edu.sg} \and Kok
Bin Wong \thanks{
Institute of Mathematical Sciences, Faculty of Science, Universiti Malaya, 50603
Kuala Lumpur, Malaysia.\newline E-mail:
kbwong@um.edu.my.} } \maketitle

\begin{abstract}\noindent
Let $S_{n}$ denotes the set of permutations of $[n]=\{1,2,\dots, n\}$. For each integer $k\geq 1$, let $S_{n,k}$  be the set of all permutations of $[n]$ with exactly $k$ disjoint cycles. 
 A subset $H\subseteq S_{n,k}$ is to be a matching if $\pi_1$ and $\pi_2$ do not have any common cycles for all distinct $\pi_1,\pi_2\in H$. The matching number of a family  $\mathcal A\subseteq S_{n,k}$ is denoted by $\nu_{p}(\mathcal A)$ and is defined to be the size of the largest  matching in $\mathcal A$.
In this paper, we determine the maximum size of a family  $\mathcal A\subseteq S_{n,k}$ subject to the condition $\nu_p(\mathcal A)\leq s$.
\end{abstract}

\bigskip\noindent
{\sc keywords:} Matchings, Independent sets, Erd{\H o}s Matching Conjecture, permutations.

\section{Introduction}

Let $[n]=\{1, \ldots, n\}$, and let $\binom{[n]}{k}$ denotes the
family of all $k$-subsets of $[n]$. A family $\mathcal{A}\subseteq \binom{[n]}{k}$ can be considered as a $k$-graph  ($k$-uniform hypergraph) on the vertex set $[n]$.
A subset $\mathcal S$ in $\mathcal A$ is called a \emph{matching} or an \emph {independent set} if $S_1\cap S_2=\varnothing$ for all $S_1,S_2\in\mathcal S$ with $S_1\neq S_2$.
The matching number $\nu(\mathcal A)$ is defined as 
\begin{align*}
\nu(\mathcal A)=\max \left\{ \vert\mathcal S\vert\ :\ \mathcal S \ \ \textnormal{is a matching in}\ \ \mathcal A  \right\}.
\end{align*}
One of the problems in extremal combinatorics is to determine the maximum size of $\mathcal A$ subject to the condition $\nu(\mathcal A)\leq s$. In other words, it is to determine
\begin{align*}
\textnormal{EMC}(n,k,s)=\max \left\{ \vert \mathcal A\vert \ :\ \mathcal A\subseteq \binom{[n]}{k}\ \ \textnormal{and} \ \ \nu(\mathcal A)\leq s  \right\}.
\end{align*}

Erd{\H o}s \cite{Erdos} made the following conjecture and showed that it is true for sufficiently large $n$, i.e., $n\geq n_0(k,s)$.

\begin{con}\label{con_erdos}\textnormal{[Erd{\H o}s Matching Conjecture]}
\begin{align*}
\textnormal{EMC}(n,k,s)=\max\left\{  \binom{n}{k}- \binom{n-s}{k}, \binom{k(s+1)-1}{k}  \right\}.
\end{align*}
\end{con}

The  Erd{\H o}s Matching Conjecture is deeply related to certain problems in probability theory dealing with generalizations of Markov’s inequality, as well as some computer science questions. This has been pointed out in \cite{Alon, Alon2}. The Erd{\H o}s Matching Conjecture is trivial for $k=1$ and the case $k=2$ was proved by Erd{\H o}s and Gallai in \cite{EG}. The case $k=3$ was also settled due to the works of Frankl \cite{Frankl3}, Frankl, R\"odl  and Ruci\'nski \cite{FRR}, and \L{}uczak and  Mieczkowska \cite{Luc}. For the general case, Bollob\'as,  Daykin, and Erd{\H o}s \cite{Bol} showed that the conjecture holds for $n\geq 2k^3s$, which is an improvement to Erd{\H o}s's original bound in \cite{Erdos}. Later, Huang, Loh, and Sudakov \cite{Huang} improved the bound further, where they showed that the conjecture holds for $n\geq 3k^2s$. Recent work of Frankl and Kupavskii in \cite{FK2} has reduced the term in the bound from quadratic in $k$ to just linear in $k$. Formally, Frankl and Kupavskii showed that there exists a $s_0$, such that for all $s\geq s_0$ and $n\geq \frac{5}{3}sk-\frac{2}{3}s$, the conjecture is true (see also \cite{Frankl1, Frankl2}). It is worth mentioning that when $s=1$, the Erd{\H o}s Matching Conjecture is the well-known Erd{\H o}s-Ko-Rado theorem in \cite{EKR}. Recently, Ku and Wong \cite{Ku_Wong8} showed an analogue of the Erd{\H o}s Matching Conjecture for weak compositions. As in the case for sets,  their result is a generalization of the analogue of Erd{\H o}s-Ko-Rado theorem for weak compositions \cite{Ku_Wong4} (see also \cite{Ku_Wong5, Ku_Wong6, Ku_Wong7} for other related results). 

 Let $S_{n}$ denotes the set of permutations of $[n]$. For a positive integer $k$, define $S_{n,k}$ to be the set of all permutations of $[n]$ with exactly $k$ disjoint cycles, i.e.,
\[ S_{n,k} = \{\pi \in S_{n}: \pi = c_{1}c_{2} \cdots c_{k}\},\] 
where $c_1,c_2,\dots ,c_k$ are disjoint cycles in $S_n$. It is well known that the size of $S_{n,k}$ is given by $\left [ \begin{matrix}n\\ k \end{matrix}\right]=(-1)^{n-k}s(n,k)$, where $s(n,k)$ is the {\em Stirling number of the first kind}.

We shall use the following notations:
\begin{itemize}
\item[\textnormal{(a)}] $N(c)=\{a_1,a_2,\dots, a_l\}$ for a cycle $c=(a_1,a_2,\dots, a_l)$;
\item[\textnormal{(b)}] $M(\pi)=\{c_1,c_2,\dots, c_k\}$ for a $\pi=c_1c_2\dots c_k\in S_{n,k}$. 
\end{itemize}

A subset $H\subseteq S_{n,k}$ is said to be a matching if $M(\pi_1)\cap M(\pi_2)=\varnothing$ for all $\pi_1,\pi_2\in H$ with $\pi_1\neq \pi_2$. Let $\mathcal A\subseteq S_{n,k}$. The size of the largest matching in $\mathcal A$ is denoted by $\nu_p(\mathcal A)$, i.e.,
\begin{align*}
\nu_p(\mathcal A)= \max \left\{ \vert H\vert \ :\ H \ \ \textnormal{ is a matching in $\mathcal A$}  \right\}.
\end{align*}
In this paper, we will show that (see Theorem \ref{thm_main}) if $\mathcal A\subseteq S_{n,k}$ satisfying $\nu_p(\mathcal A)\leq s$, then
\begin{align*}
\vert \mathcal A\vert \leq  \sum_{1\leq i\leq s} (-1)^{i-1} \binom{s}{i}\left[ \begin{matrix}n-i\\ k-i \end{matrix}\right].
\end{align*}
Now, if we set
\begin{align*}
\textnormal{EMC}_p(n,k,s)=\max \left\{ \vert \mathcal A\vert \ :\ \mathcal A\subseteq S_{n,k}\ \ \textnormal{and} \ \ \nu_p(\mathcal A)\leq s  \right\},
\end{align*}
then (see Corollary \ref{cor_main})
\begin{align}
\textnormal{EMC}_p(n,k,s)=\sum_{1\leq i\leq s} (-1)^{i-1} \binom{s}{i}\left[ \begin{matrix}n-i\\ k-i \end{matrix}\right].\label{eq_emc_1}
\end{align}
In particular, 
\begin{align}
\textnormal{EMC}_p(n,k,1)=\left[ \begin{matrix}n-1\\ k-1 \end{matrix}\right].\label{eq_emc_2}
\end{align}

A family $\mathcal{A} \subseteq S_{n,k}$ is said to be {\em cycle-intersecting} if any two elements of $\mathcal{A}$ have at least a common cycle, i.e., $M(\pi_1)\cap M(\pi_2)\neq\varnothing$ for all $\pi_1,\pi_2\in \mathcal A$. Note that if $\nu_p(\mathcal A)=0$, then $\mathcal A=\varnothing$ and if  $\nu_p(\mathcal A)=1$, then $\mathcal A$ is cycle-intersecting. Therefore, if $\nu_p(\mathcal A)\leq 1$ and $\mathcal A\neq \varnothing$, then $\mathcal A$ must be cycle-intersecting.
Note that equation (\ref{eq_emc_2}) follows from the following theorem which is a special case of \cite[Theorem 2]{Ku_Wong41}.

\begin{thm}\label{thm_erk_permutation} Suppose $k\geq 2$. There exists a function $n_0(k)=O(k^{3})$ such that  if $\mathcal{A} \subseteq S_{n,k}$ is cycle-intersecting and $n \ge n_{0}(k)$, then 
\[ |\mathcal{A}| \le \left [ \begin{matrix}n-1\\ k-1 \end{matrix}\right].\]
\end{thm}

Theorem  \ref{thm_erk_permutation} is also well-known as the analogue of Erd{\H o}s-Ko-Rado theorem for permutations with fixed number of cycles (see also \cite{Ellis1,Ellis2, Ku, Ku_Wong61} for other related results).

\section{Main result}

The following recurrence relation of the unsigned Stirling number $\left [ \begin{matrix}n\\ k \end{matrix}\right]$ is well-known:
\begin{equation}\label{eq_recurrence}
\left [ \begin{matrix}n\\ k \end{matrix}\right]=\left [ \begin{matrix}n-1\\ k-1 \end{matrix}\right]+(n-1)\left [ \begin{matrix}n-1\\ k \end{matrix}\right],
\end{equation}
with initial conditions $\left [ \begin{matrix}0\\ 0 \end{matrix}\right] = 1$ and $\left [ \begin{matrix}n\\ 0 \end{matrix}\right] = \left [ \begin{matrix} 0\\ k \end{matrix}\right] = 0$,  $n > 0$. Note that $\left [ \begin{matrix}n\\ n\end{matrix}\right]=1$ and
\begin{align}
\left [ \begin{matrix}n\\ k \end{matrix}\right]\geq \left [ \begin{matrix}n-1\\ k-1 \end{matrix}\right]\qquad \textnormal{and}\qquad \left [ \begin{matrix}n\\ k \end{matrix}\right]\geq \left [ \begin{matrix}n-1\\ k \end{matrix}\right].\label{eq_inequality}
\end{align}
We will use the two inequalities (\ref{eq_inequality})  without mentioning them. We shall use the conventions that $\left [ \begin{matrix}n\\ k \end{matrix}\right]=0$ if $k<0$ or $n<k$. By using equation (\ref{eq_recurrence}), it is not hard to obtain
\begin{align}
\left [ \begin{matrix} n\\ 1 \end{matrix}\right] =(n-1)!\qquad\textnormal{and}\qquad\left [ \begin{matrix} k+1\\ k \end{matrix}\right] =\frac{(k+1)k}{2}.\label{eq_inequality2}
\end{align}

\begin{lm}\label{lm_stirling_inequality} If $n\geq k+1$, then
\begin{align*}
\left [ \begin{matrix} n\\ k+1 \end{matrix}\right] \geq \frac{1}{k^2}\left [ \begin{matrix} n\\ k \end{matrix}\right].
\end{align*}
\end{lm}

\begin{proof} When $n=k+1$, $\left [ \begin{matrix} n\\ k+1 \end{matrix}\right] =1$ and by equation (\ref{eq_inequality2}), $$\frac{1}{k^2}\left [ \begin{matrix} n\\ k \end{matrix}\right]=\frac{(k+1)k}{2k^2}\leq 1=\left [ \begin{matrix} n\\ k+1 \end{matrix}\right].$$ So, the lemma holds for $n=k+1$. Assume the lemma holds for all $n'$ with $k+1\leq n'<n$. By equation (\ref{eq_recurrence}),
\begin{align*}
\left [ \begin{matrix}n\\ k+1 \end{matrix}\right]&=\left [ \begin{matrix}n-1\\ k \end{matrix}\right]+(n-1)\left [ \begin{matrix}n-1\\ k+1 \end{matrix}\right]\\
&\geq \frac{1}{(k-1)^2}\left [ \begin{matrix}n-1\\ k-1 \end{matrix}\right]+\frac{n-1}{k^2}\left [ \begin{matrix}n-1\\ k \end{matrix}\right]\\
&> \frac{1}{k^2}\left [ \begin{matrix}n-1\\ k-1 \end{matrix}\right]+\frac{n-1}{k^2}\left [ \begin{matrix}n-1\\ k \end{matrix}\right]\\
&= \frac{1}{k^2}\left(\left [ \begin{matrix}n-1\\ k-1 \end{matrix}\right]+(n-1)\left [ \begin{matrix}n-1\\ k \end{matrix}\right]\right)\\
&= \frac{1}{k^2}\left [ \begin{matrix}n\\ k \end{matrix}\right].
\end{align*}
So, it is true for $n$ and by induction, the lemma holds.
\end{proof}

\begin{lm}\label{lm_stirling_inequality2} If $n\geq k+1$ and $k\geq 2$, then
\begin{align*}
\left [ \begin{matrix} n\\ k\end{matrix}\right] > \frac{n}{k^2}\left [ \begin{matrix} n-1\\ k-1 \end{matrix}\right].
\end{align*}
\end{lm}

\begin{proof} By equation (\ref{eq_recurrence}) and Lemma \ref{lm_stirling_inequality},
\begin{align*}
\left [ \begin{matrix}n\\ k \end{matrix}\right]&=\left [ \begin{matrix}n-1\\ k-1 \end{matrix}\right]+(n-1)\left [ \begin{matrix}n-1\\ k \end{matrix}\right]\\
&\geq \left [ \begin{matrix}n-1\\ k-1 \end{matrix}\right]+\frac{n-1}{(k-1)^2}\left [ \begin{matrix}n-1\\ k-1 \end{matrix}\right]\\
&= \left [ \begin{matrix}n-1\\ k-1 \end{matrix}\right]\left(\frac{n}{(k-1)^2}+1-\frac{1}{(k-1)^2}\right)\\
&\geq  \frac{n}{(k-1)^2}\left [ \begin{matrix}n-1\\ k-1 \end{matrix}\right].
\end{align*}
The lemma follows by noting that $(k-1)^2<k^2$.
\end{proof}

\begin{lm}\label{lm_main} Let $k\geq 3$, $l\geq 1$  and $n\geq l^2+1$.   Let $b$ be a cycle in $S_n$ and
\begin{align*}
\mathcal A\subseteq \left\{\pi \in S_{n,k}\ :\ b\in M(\pi)   \right\}.
\end{align*}
Let $c_1,c_2,\dots, c_l$ be disjoint cycles in $S_n$ and $c_i\neq b$ for all $i$. 
If $\vert \mathcal A\vert\geq \sqrt{n}\left [ \begin{matrix}n-2\\ k-2 \end{matrix}\right]$, then  
there exists a $\pi_0\in \mathcal A$ such that $c_i\notin M(\pi_0)$ for all $i$.
\end{lm}

\begin{proof} For each $1\leq i\leq l$, set
\begin{align*}
\mathcal A(c_i)=\left\{ \pi\in\mathcal A\ :\ c_i\in M(\pi) \right\}.
\end{align*}
Then
\begin{align*}
\mathcal A(c_i)\subseteq \left\{\pi\in S_{n,k}\ :\ c_i,b\in M(\pi)  \right\}.
\end{align*}
Therefore,
\begin{align*}
\vert \mathcal A(c_i)\vert \leq \left [ \begin{matrix}n-\vert N(b)\vert-\vert N(c_i)\vert\\ k-2 \end{matrix}\right]\leq \left [ \begin{matrix}n-2\\ k-2 \end{matrix}\right].
\end{align*}
So,
\begin{align*}
\left\vert\bigcup_{1\leq i\leq l} \mathcal A(c_i)  \right\vert&\leq \sum_{1\leq i\leq l} \vert \mathcal A(c_i)\vert\leq \sum_{1\leq i\leq l} \left [ \begin{matrix}n-2\\ k-2 \end{matrix}\right]= l \left [ \begin{matrix}n-2\\ k-2 \end{matrix}\right].
\end{align*}

Since $n\geq  l^2+1$, we have
\begin{align*}
\left\vert\bigcup_{1\leq i\leq l} \mathcal A(c_i)  \right\vert\leq l\left [ \begin{matrix}n-2\\ k-2 \end{matrix}\right]< \sqrt{n}\left [ \begin{matrix}n-2\\ k-2 \end{matrix}\right]\leq \vert \mathcal A\vert.
\end{align*}
Thus, there is a $\pi_0\in \mathcal A\setminus \left( \bigcup_{1\leq i\leq l} \mathcal A(c_i)    \right)$. If $c_{j_0}\in M(\pi_0)$ for some $j_0$, then $\pi_0\in  \mathcal A(c_{j_0})$, a contradiction. Hence, $c_i\notin M(\pi_0)$ for all $i$. 

 This completes the proof of the lemma.
\end{proof}

\begin{lm}\label{lm_main2} Let $k\geq 3$, $r\geq 1$ and $n\geq (rk)^2+1$. Let $H\subseteq S_{n,k}$ be a matching of size $r$. Let  $b$ be a cycle in $S_n$ and
\begin{align*}
\mathcal A\subseteq \left\{\pi \in S_{n,k}\ :\ b\in M(\pi)   \right\}.
\end{align*}
 Suppose that  $b\notin M(\pi)$ for all $\pi \in H$.
If $\vert \mathcal A\vert\geq  \sqrt{n}\left [ \begin{matrix}n-2\\ k-2 \end{matrix}\right]$, then  
there exists a $\pi_0\in \mathcal A$ such that $\{\pi_0\}\cup H$ is a matching.
\end{lm}

\begin{proof}  Let's consider the set
\begin{align*}
\bigcup_{\pi\in H} M(\pi).
\end{align*}
Note that $\vert M(\pi)\vert=k$ for all $\pi\in H$. Since $H$ is a matching, $M(\pi)\cap M(\pi')=\varnothing$ for all $\pi,\pi'\in H$ with $\pi\neq \pi'$. Thus,
$\left\vert \bigcup_{\pi\in H} M(\pi)\right\vert=rk$. Let $c_1,c_2,\dots, c_{rk}$ be all the elements in $\bigcup_{\pi\in H} M(\pi)$. Note that each $c_i$ is a cycle in $S_n$. By Lemma \ref{lm_main}, there exists a $\pi_0\in \mathcal A$ such that $c_i\notin M(\pi_0)$ for all $i$.

Suppose $\{\pi_0\}\cup H$ is not a matching. Then $M(\pi_0)\cap M(\pi_1)\neq\varnothing$ for some $\pi_1\in H$. This implies that $c_{j_1}\in M(\pi_0)\cap M(\pi_1)$ for some $1\leq j_1\leq rk$. But this is impossible as $c_i\notin M(\pi_0)$ for all $i$.  Hence,  $\{\pi_0\}\cup H$  is a matching.
\end{proof}

\begin{lm}\label{lm_main3} Let $k\geq 3$, $l\geq 1$, $r\geq 1$  and $n\geq ((r+l-1)k)^2+1$. Let $H\subseteq S_{n,k}$ be a matching of size $r$. Let 
\begin{align*}
\mathcal A_1&\subseteq \left\{\pi\in S_{n,k}\ :\ b_1\in M(\pi)   \right\};\\
\mathcal A_2&\subseteq \left\{\pi\in S_{n,k}\ :\ b_2\in M(\pi)   \right\};\\
&\vdots\\
\mathcal A_l&\subseteq \left\{\pi\in S_{n,k}\ :\ b_l\in M(\pi)    \right\},
\end{align*}
where each $b_i$ is a cycle in $S_n$, $b_i\neq b_{i'}$ for all $i\neq i'$ and $b_i\notin M(\pi)$ for all $\pi\in H$ and all $i$.
If $\vert \mathcal A_i\vert\geq \sqrt{n}\left [ \begin{matrix}n-2\\ k-2 \end{matrix}\right]$ for all $1\leq i\leq l$, then  
there exists a  $\pi_i\in \mathcal A_i$ for each $i$, such that $\{\pi_1,\pi_2,\dots, \pi_l\}\cup H$ is a matching.
\end{lm}

\begin{proof}  We shall prove by induction on $l$. The case $l=1$ follows from Lemma \ref{lm_main2}. Suppose $l\geq 2$. Assume that the lemma holds for all $1\leq l'< l$. 

Let $c_1,c_2,\dots, c_{rk}$ be all the cycles in $\bigcup_{\pi\in H} M(\pi)$. Note that $c_j\neq b_i$ for all $1\leq i\leq l$ and $1\leq j\leq rk$. Now, let $c_{rk+i}=b_i$ for all $1\leq i\leq l-1$. Then $b_l\neq c_j$ for all $1\leq j\leq rk+l-1$. By Lemma \ref{lm_main}, there exists a $\pi_l\in \mathcal A_l$ such that $c_j\notin M(\pi_l)$ for all $1\leq j\leq rk+l-1$, provided that $n\geq (rk+l-1)^2+1$, which is true as $n\geq ((r+l-1)k)^2+1$. 

Now, $H'=\{\pi_l\}\cup H$ is a matching of size $r+1$. Clearly, $b_i\notin M(\pi)$ for all $\pi\in H'$ and all $1\leq i\leq l-1$. By induction, there exists a  $\pi_i\in \mathcal A_i$ for $1\leq i\leq l-1$ such that $\{\pi_1,\dots, \pi_{l-1}\}\cup H'$ is a matching, provided that $n\geq (((r+1)+l-2)k)^2+1=((r+l-1)k)^2+1$.
Hence, $\{\pi_1,\pi_2,\dots, \pi_l\}\cup H$ is a matching.
\end{proof}

\begin{lm}\label{lm_main_eq3} Let $k,s$ be positive integers with $k\geq 3$. If $n\geq  sk^2$, then
\begin{align*}
\sum_{1\leq i\leq s} (-1)^{i-1} \binom{s}{i}\left[ \begin{matrix}n-i\\ k-i \end{matrix}\right]\geq s\left[ \begin{matrix}n-1\\ k-1 \end{matrix}\right]-\frac{s(s-1)}{2}\left[ \begin{matrix}n-2\\ k-2 \end{matrix}\right]
\end{align*}
\end{lm}

\begin{proof} Clearly, the lemma holds if $s=1$. So, we may assume that $s\geq 2$. Let $a_i=\displaystyle\binom{s}{i}\left[ \begin{matrix}n-i\\ k-i \end{matrix}\right]$. 
If $a_{i+1}=0$ and $a_i\neq 0$, then $a_{i+1}=0<1\leq a_i$. 
If $a_{i+1}=0$ and $a_i= 0$, then $a_{i+1}=0=a_i$.  Note that if $a_{i+1}=0$, then $k\leq i+1$. Suppose $k>i+1$. Then $a_{i+1}\neq 0$ and $a_i\neq 0$. By Lemma \ref{lm_stirling_inequality2},  
\begin{align*}
\left [ \begin{matrix} n-i\\ k-i\end{matrix}\right] > \frac{n-i}{(k-i)^2}\left [ \begin{matrix} n-i-1\\ k-i-1 \end{matrix}\right].
\end{align*}
Therefore,
\begin{align*}
\frac{a_{i+1}}{a_i} &=\frac{\displaystyle\binom{s}{i+1}\left[ \begin{matrix}n-i-1\\ k-i-1 \end{matrix}\right]}{\displaystyle\binom{s}{i}\left[ \begin{matrix}n-i\\ k-i \end{matrix}\right]}\leq \left(\frac{s-i}{i+1}\right)\left(   \frac{(k-i)^2} {n-i}\right)\\
&\leq \left(\frac{s-1}{2}\right)\left(   \frac{(k-1)^2} {n-s+1}\right)<1,
\end{align*}
provided that $n\geq sk^2>\frac{(s-1)(k-1)^2}{2}+s-1$.

Now, $\sum_{1\leq i\leq s} (-1)^{i-1} a_i$ is an alternating series with $a_{i+1}\leq a_i$. Hence, $\sum_{1\leq i\leq s} (-1)^{i-1} a_i\geq a_1-a_2$, and the lemma follows.
\end{proof}

\begin{lm}\label{lm_main_eq2} Let $b_1,b_2,\dots, b_s$ be distinct cycles in $S_n$. For each $1\leq t\leq s$, let
\begin{align*}
\mathcal A_t=\left\{ \pi\in S_{n,k}\ :\ b_t\in M(\pi)  \right\}. 
\end{align*}
Suppose $k\geq 3$ and $n\geq  (sk)^2$. If 
\begin{align*}
\mathcal A=\bigcup_{1\leq t\leq s} \mathcal A_t,
\end{align*}
then
\begin{align*}
\vert \mathcal A\vert\leq \sum_{1\leq i\leq s} (-1)^{i-1} \binom{s}{i}\left[ \begin{matrix}n-i\\ k-i \end{matrix}\right].
\end{align*}
Furthermore,  equality holds if and only if $\vert N(b_t)\vert=1$ for all $t$, i.e., each $b_t$ is a cycle of length 1.
\end{lm}

\begin{proof} Suppose each $b_t$ is a cycle of length 1.  Then  all the $b_i$'s are disjoint cycles. By the Principle of Inclusion and Exclusion, 
\begin{align*}
\vert \mathcal A\vert &=\left\vert \bigcup_{1\leq t\leq s} \mathcal A_t\right\vert=\sum_{1\leq i\leq s} (-1)^{i-1} \sum_{\substack{V\subseteq [s],\\ \vert V\vert=i}} \left\vert \bigcap_{v\in V} \mathcal A_v \right\vert.
\end{align*}
Now,
\begin{align*}
 \bigcap_{v\in V} \mathcal A_v  =\left\{  \pi\in S_{n,k}\ :\ b_v\in M(\pi)\ \ \textnormal{for all $v\in V$}  \right\}.
\end{align*}
Therefore,
\begin{equation}
\left\vert \bigcap_{v\in V} \mathcal A_v \right\vert =
\left[ \begin{matrix}n-\sum_{v\in V} \vert N(b_v)\vert\\ k-\vert V\vert \end{matrix}\right]=\left[ \begin{matrix}n-\vert V\vert\\ k-\vert V\vert \end{matrix}\right], \notag
\end{equation}
and
\begin{align*}
\vert \mathcal A\vert=\sum_{1\leq i\leq s} (-1)^{i-1} \binom{s}{i}\left[ \begin{matrix}n-i\\ k-i \end{matrix}\right].
\end{align*}

Suppose $\vert N(b_t)\vert>1$ for some $t$. Without loss of generality, we may assume that $\vert N(b_1)\vert>1$, i.e., $b_1$ is a cycle of length at least 2. Now,
\begin{align*}
\vert \mathcal A\vert &=\left\vert \bigcup_{1\leq t\leq s} \mathcal A_t \right\vert \leq  \sum_{1\leq t\leq s} \left\vert \mathcal A_t \right\vert =\sum_{1\leq t\leq s} \left[ \begin{matrix}n-\vert N(b_t)\vert\\ k-1 \end{matrix}\right]\\
&=\left[ \begin{matrix}n-\vert N(b_1)\vert\\ k-1 \end{matrix}\right]+\sum_{2\leq t\leq s} \left[ \begin{matrix}n-\vert N(b_t)\vert\\ k-1 \end{matrix}\right]\\
&\leq \left[ \begin{matrix}n-2\\ k-1 \end{matrix}\right]+\sum_{2\leq t\leq s} \left[ \begin{matrix}n-1\\ k-1 \end{matrix}\right]\\
&= \left[ \begin{matrix}n-2\\ k-1 \end{matrix}\right]+(s-1)\left[ \begin{matrix}n-1\\ k-1 \end{matrix}\right].
\end{align*}

By Lemma \ref{lm_stirling_inequality2},
\begin{align*}
\left [ \begin{matrix} n-1\\ k-1\end{matrix}\right] > \frac{n-1}{(k-1)^2}\left [ \begin{matrix} n-2\\ k-2 \end{matrix}\right].
\end{align*}
Therefore, by Lemma \ref{lm_main_eq3}, 
\begin{align*}
\sum_{1\leq i\leq s} (-1)^{i-1} \binom{s}{i}\left[ \begin{matrix}n-i\\ k-i \end{matrix}\right]&\geq s\left[ \begin{matrix}n-1\\ k-1 \end{matrix}\right]-\frac{s(s-1)}{2}\left[ \begin{matrix}n-2\\ k-2 \end{matrix}\right]\\
&= (s-1)\left[ \begin{matrix}n-1\\ k-1 \end{matrix}\right]+\left(\left[ \begin{matrix}n-1\\ k-1 \end{matrix}\right]-\frac{s(s-1)}{2}\left[ \begin{matrix}n-2\\ k-2 \end{matrix}\right]\right)\\
&> (s-1)\left[ \begin{matrix}n-1\\ k-1 \end{matrix}\right]+\left(\frac{n-1}{(k-1)^2}-\frac{s(s-1)}{2}\right)\left[ \begin{matrix}n-2\\ k-2 \end{matrix}\right]\\
&>(s-1)\left[ \begin{matrix}n-1\\ k-1 \end{matrix}\right]+\left[ \begin{matrix}n-2\\ k-1 \end{matrix}\right]\geq \vert \mathcal A\vert ,
\end{align*}
provided that $\frac{n-1}{(k-1)^2}-\frac{s(s-1)}{2}>1$, which is true when $n\geq (sk)^2$.

This completes the proof of the lemma.
\end{proof}

\begin{thm}\label{thm_main} Let $k\geq 4$, $s\geq 1$ and $n\geq 4s^2k^6$. If $\mathcal A\subseteq S_{n,k}$ with $\nu_p(\mathcal A)\leq s$, then
\begin{align*}
\vert \mathcal A\vert\leq \sum_{1\leq i\leq s} (-1)^{i-1} \binom{s}{i}\left[ \begin{matrix}n-i\\ k-i \end{matrix}\right].
\end{align*}
Moreover,  equality holds if and only if $\nu_p(\mathcal A)=s$ and there exists a $s$-set $T\subseteq [n]$ such that
\begin{align*}
\mathcal A=\bigcup_{t\in T} \mathcal A_t,
\end{align*}
where
\begin{align*}
\mathcal A_t=\left\{ \pi\in S_{n,k}\ :\ (t)\in M(\pi)  \right\}. 
\end{align*}
\end{thm}

\begin{proof} Let $\{\pi_1,\pi_2,\dots, \pi_r\}$ be a matching in $\mathcal A$ with $\nu_p(\mathcal A)=r$. This means for each $\pi\in\mathcal A$, $M(\pi)\cap M(\pi_j)\neq \varnothing$ for some $j\in [r]$. Since $\nu_p(\mathcal A)\leq s$, we have $r\leq s$. Let
\begin{align*}
\pi_1 &=c_{11}c_{12}\cdots c_{1k};\\
\pi_2 &=c_{21}c_{22}\cdots c_{2k};\\
&\vdots\\
\pi_r &=c_{r1}c_{r2}\cdots c_{rk},
\end{align*}
be their cycle decompositions in $S_{n,k}$. Now, for each $1\leq i\leq r$ and $1\leq j\leq k$, let
\begin{align*}
\mathcal A_{ij}=\left\{ \pi\in \mathcal A\ :\ c_{ij}\in M(\pi)\right\}.
\end{align*}
Then $$\mathcal A=\bigcup_{\substack{1\leq i\leq r,\\ 1\leq j\leq k}} \mathcal A_{ij}.$$
We distinguish two cases.

\vskip 0.5cm
\noindent
{\bf Case 1.} Suppose $\vert \mathcal A_{ij}\vert <\sqrt{n}\left [ \begin{matrix}n-2\\ k-2 \end{matrix}\right]$ for all $i,j$.

Note that 
\begin{align*}
\vert \mathcal A\vert &\leq \sum_{\substack{1\leq i\leq r,\\ 1\leq j\leq k}} \vert \mathcal A_{ij}\vert<\sum_{\substack{1\leq i\leq r,\\ 1\leq j\leq k}}\sqrt{n}\left [ \begin{matrix}n-2\\ k-2 \end{matrix}\right]=rk\sqrt{n}\left [ \begin{matrix}n-2\\ k-2 \end{matrix}\right]\leq sk\sqrt{n}\left [ \begin{matrix}n-2\\ k-2 \end{matrix}\right].
\end{align*}
On the other hand, by Lemma \ref{lm_main_eq3},
\begin{align*}
\sum_{1\leq i\leq s} (-1)^{i-1} \binom{s}{i}\left[ \begin{matrix}n-i\\ k-i \end{matrix}\right]&\geq s\left[ \begin{matrix}n-1\\ k-1 \end{matrix}\right]-\frac{s(s-1)}{2}\left[ \begin{matrix}n-2\\ k-2 \end{matrix}\right]\\
&>s\left[ \begin{matrix}n-1\\ k-1 \end{matrix}\right]-s\sqrt{n}\left[ \begin{matrix}n-2\\ k-2 \end{matrix}\right],
\end{align*}
provided that $n\geq s^2$. Therefore, $$\vert \mathcal A\vert <\sum_{1\leq i\leq s} (-1)^{i-1} \binom{s}{i}\left[ \begin{matrix}n-i\\ k-i \end{matrix}\right],$$ if and only if
\begin{align}
\left[ \begin{matrix}n-1\\ k-1 \end{matrix}\right]\geq (k+1)\sqrt{n}\left[ \begin{matrix}n-2\\ k-2 \end{matrix}\right].\label{eq_thm_main}
\end{align}
By Lemma  \ref{lm_stirling_inequality2},
\begin{align*}
\left [ \begin{matrix} n-1\\ k-1\end{matrix}\right] > \frac{n-1}{(k-1)^2}\left [ \begin{matrix} n-2\\ k-2 \end{matrix}\right].
\end{align*}
Thus, equation (\ref{eq_thm_main}) follows by noticing that
\begin{align*}
 \frac{n-1}{(k-1)^2}\geq (k+1)\sqrt{n},
\end{align*}
which is equivalent to $\sqrt{n}-\frac{1}{\sqrt{n}}\geq (k^2-1)(k-1)$ and this is true provided that $n\geq k^6$.

\vskip 0.5cm
\noindent
{\bf Case 2.} Suppose $\vert \mathcal A_{ij}\vert \geq \sqrt{n}\left [ \begin{matrix}n-2\\ k-2 \end{matrix}\right]$ for some $i,j$. 

We claim that for each $i$, there is at most one $j$ with $\vert \mathcal A_{ij}\vert \geq \sqrt{n}\left [ \begin{matrix}n-2\\ k-2 \end{matrix}\right]$. Suppose 
 $\vert \mathcal A_{ij}\vert \geq\sqrt{n}\left [ \begin{matrix}n-2\\ k-2 \end{matrix}\right]$ and  $\vert \mathcal A_{ij'}\vert \geq \sqrt{n}\left [ \begin{matrix}n-2\\ k-2 \end{matrix}\right]$ for some $j\neq j'$.

Let $H=\{\pi_x \ :\ 1\leq x\leq r, x\neq i\}$. Then $H$ is a matching of size $r-1$. Now, 
\begin{align*}
\mathcal A_{ij}\subseteq \left\{\pi\in S_{n,k}\ :\ c_{ij}\in M(\pi)  \right\};\\
\mathcal A_{ij'}\subseteq \left\{\pi\in S_{n,k}\ :\ c_{ij'}\in M(\pi)  \right\}.
\end{align*}
Since $\{\pi_1,\pi_2,\dots, \pi_r\}$ is a matching, we have $c_{ij},c_{ij'}\notin M(\pi)$ for all $\pi\in H$.

Noting that $n\geq (sk)^2+1\geq (rk)^2+1$, it follows from Lemma \ref{lm_main3}  that there exist $\pi_1'\in \mathcal A_{ij}$ and $\pi_2'\in \mathcal A_{ij'}$  
 such that $\{\pi_1',\pi_2'\}\cup H$ is a matching. This means $\mathcal A$ has a matching of size $r+1$, which is not possible as $\nu_p(\mathcal A)=r$. 
Hence, we have shown that for each $i$, there is at most one $j$ for which  $\vert \mathcal A_{ij}\vert \geq \sqrt{n}\left [ \begin{matrix}n-2\\ k-2 \end{matrix}\right]$.

By relabelling if necessary, we may assume that for each $1\leq i\leq l$, there is a unique $j_i$ such that $\vert \mathcal A_{ij_i}\vert \geq \sqrt{n}\left [ \begin{matrix}n-2\\ k-2 \end{matrix}\right]$ and for each $l+1\leq i\leq r$, $\vert \mathcal A_{ij}\vert < \sqrt{n}\left [ \begin{matrix}n-2\\ k-2 \end{matrix}\right]$ for all $j$. 

Next, we show that
\begin{align}
\mathcal A=\bigcup_{\substack{1\leq i\leq l}} \mathcal A_{ij_i} \cup \bigcup_{\substack{l+1\leq i\leq r,\\ 1\leq j\leq k}} \mathcal A_{ij}.\label{eq_main_set}
\end{align}
Clearly, 
\begin{align*}
\bigcup_{\substack{1\leq i\leq l}} \mathcal A_{ij_i} \cup \bigcup_{\substack{l+1\leq i\leq r,\\ 1\leq j\leq k}} \mathcal A_{ij}\subseteq \mathcal A.
\end{align*}
It is sufficient to show that 
\begin{align*}
\mathcal A\subseteq \bigcup_{\substack{1\leq i\leq l}} \mathcal A_{ij_i} \cup \bigcup_{\substack{l+1\leq i\leq r,\\ 1\leq j\leq k}} \mathcal A_{ij}.
\end{align*}
 Let $\alpha\in \mathcal A$. Suppose $\alpha\notin \bigcup_{\substack{1\leq i\leq l}} \mathcal A_{ij_i} \cup \bigcup_{\substack{l+1\leq i\leq r, 1\leq j\leq k}} \mathcal A_{ij}$. Now, $\alpha\notin \bigcup_{\substack{l+1\leq i\leq r, 1\leq j\leq k}} \mathcal A_{ij}$ implies that
\begin{align*}
H_1=\{\alpha, \pi_{l+1},\pi_{l+2},\dots, \pi_r\},
\end{align*}
is a matching of size $r-l+1$. 

Note that
\begin{align*}
\mathcal A_{1j_1}&\subseteq \left\{\pi\in S_{n,k}\ :\ c_{1j_1}\in M(\pi)  \right\};\\
\mathcal A_{2j_2}&\subseteq \left\{\pi\in S_{n,k}\ :\ c_{1j_2}\in M(\pi)     \right\};\\
&\vdots\\
\mathcal A_{lj_l}&\subseteq \left\{\pi\in S_{n,k}\ :\ c_{1j_l}\in M(\pi)    \right\}.
\end{align*}
Since $\{\pi_1,\pi_2,\dots, \pi_r\}$ is a matching, $c_{ij_i}\neq c_{i'j_{i'}}$ for all $i\neq i'$ and  $c_{ij_i}\notin M(\pi)$ for all $\pi\in H_1\setminus \{\alpha\}$ and all $i$. 
It follows from  $\alpha\notin \bigcup_{\substack{1\leq i\leq l}} \mathcal A_{ij_i}$ that $c_{ij_i}\notin M(\alpha)$. Thus, for each $i$, $c_{ij_i}\notin M(\pi)$ for all $\pi\in H_1$. 
 By Lemma \ref{lm_main3} and $n\geq (sk)^2+1\geq (rk)^2+1=(((r-l+1)+l-1)k)^2+1$, there exist $\mathbf \beta_i\in \mathcal A_{ij_i}$ for $1\leq i\leq l$, such that 
$$\{\beta_1,\beta_2,\dots, \beta_l\}\cup H_1=\{\beta_1,\beta_2,\dots, \beta_l,\alpha, \pi_{l+1},\pi_{l+2},\dots, \pi_r\},$$ is a matching of size $r+1$, contradicting  $\nu_p(\mathcal A)=r$. Hence,  (\ref{eq_main_set}) must be true.

We consider two cases.

\vskip 0.5cm
\noindent
{\bf Case 2.1.} Suppose $1\leq l\leq r-1$. Then $2\leq r\leq s$. 

Note that
\begin{align*}
\left\vert \bigcup_{\substack{1\leq i\leq l}} \mathcal A_{ij_i}\right\vert&\leq  \sum_{\substack{1\leq i\leq l}} \left\vert\mathcal A_{ij_i}\right\vert\leq  \sum_{\substack{1\leq i\leq l}} \left [ \begin{matrix} n-\vert N(c_{ij_i})\vert\\ k-1\end{matrix}\right] \leq  \sum_{\substack{1\leq i\leq l}} \left [ \begin{matrix} n-1\\ k-1\end{matrix}\right]\\
& =l\left [ \begin{matrix} n-1\\ k-1\end{matrix}\right]\leq (r-1)\left [ \begin{matrix} n-1\\ k-1\end{matrix}\right]\leq (s-1)\left [ \begin{matrix} n-1\\ k-1\end{matrix}\right].
\end{align*}
On the other hand,
\begin{align*}
\left\vert \bigcup_{\substack{l+1\leq i\leq r,\\ 1\leq j\leq k}} \mathcal A_{ij} \right\vert<(r-l)k\sqrt{n}\left [ \begin{matrix}n-2\\ k-2 \end{matrix}\right]<sk\sqrt{n}\left [ \begin{matrix}n-2\\ k-2 \end{matrix}\right].
\end{align*}
Therefore, by (\ref{eq_main_set}) and Lemmas  \ref{lm_stirling_inequality2} and \ref{lm_main_eq3},
\begin{align*}
\left\vert \mathcal A\right\vert&\leq \left\vert \bigcup_{\substack{1\leq i\leq l}} \mathcal A_{ij_i}\right\vert+\left\vert \bigcup_{\substack{l+1\leq i\leq r,\\ 1\leq j\leq k}} \mathcal A_{ij} \right\vert<(s-1)\left [ \begin{matrix} n-1\\ k-1\end{matrix}\right]+sk\sqrt{n}\left [ \begin{matrix}n-2\\ k-2 \end{matrix}\right]\\
&=s\left [ \begin{matrix} n-1\\ k-1\end{matrix}\right]+\left(sk\sqrt{n}\left [ \begin{matrix}n-2\\ k-2 \end{matrix}\right]-\left [ \begin{matrix} n-1\\ k-1\end{matrix}\right]\right)<s\left [ \begin{matrix} n-1\\ k-1\end{matrix}\right]+\left(sk\sqrt{n}-\frac{n-1}{(k-1)^2}\right)\left [ \begin{matrix}n-2\\ k-2 \end{matrix}\right]\\
&\leq s\left [ \begin{matrix} n-1\\ k-1\end{matrix}\right]-sk\sqrt{n}\left [ \begin{matrix}n-2\\ k-2 \end{matrix}\right]<s\left[ \begin{matrix}n-1\\ k-1 \end{matrix}\right]-\frac{s(s-1)}{2}\left[ \begin{matrix}n-2\\ k-2 \end{matrix}\right]\\
&\leq \sum_{1\leq i\leq s} (-1)^{i-1} \binom{s}{i}\left[ \begin{matrix}n-i\\ k-i \end{matrix}\right],
\end{align*}
provided that $sk\sqrt{n}-\frac{n-1}{(k-1)^2}\leq -sk\sqrt{n}$, which holds for $n\geq 4s^2k^6$.

\vskip 0.5cm
\noindent
{\bf Case 2.2.} Suppose $l=r$. 

Then (\ref{eq_main_set}) becomes
\begin{align*}
\mathcal A=\bigcup_{\substack{1\leq i\leq r}} \mathcal A_{ij_i}.
\end{align*}
By Lemma \ref{lm_main_eq2}, 
\begin{align*}
\vert \mathcal A\vert\leq \sum_{1\leq i\leq s} (-1)^{i-1} \binom{s}{i}\left[ \begin{matrix}n-i\\ k-i \end{matrix}\right],
\end{align*}
and equality holds if and only if $\vert N(b_t)\vert=1$ for all $t$, i.e., each $b_t$ is a cycle of length 1.

This completes the proof of the theorem.
\end{proof}

The following corollary is an immediate consequence of Theorem \ref{thm_main}.

\begin{cor}\label{cor_main} Let $k\geq 4$ and  $s\geq 1$. If $n\geq 4s^2k^6$, then 
\begin{align*}
\textnormal{EMC}_p(n,k,s)=\sum_{1\leq i\leq s} (-1)^{i-1} \binom{s}{i}\left[ \begin{matrix}n-i\\ k-i \end{matrix}\right].
\end{align*}
\end{cor}


{\footnotesize
\begin{thebibliography}{99}\setlength{\itemsep}{-1mm}


\bibitem{Alon} N. Alon, P. Frankl, H. Huang, V. R{\H o}dl, A. Ruci\'nski, and  B. Sudakov, Large matchings in uniform hypergraphs and the conjectures of Erd{\H o}s and Samuels, {\em J. Comb. Theory, Ser. A} {\bf 119} (2012) 1200--1215.

\bibitem{Alon2} N. Alon, H. Huang, and B. Sudakov, Nonnegative k-sums, fractional covers, and probability of small deviations, {\em J. Comb. Theory, Ser. B} {\bf 102} (2012) 784--796.



\bibitem{Bol} B. Bollob\'as, D.E. Daykin, and P. Erd{\H o}s, Sets of independent edges of a hypergraph, {\em Q. J. Math.} {\bf 27} (1976) 25--32.

\bibitem{Ellis1} D. Ellis, Stability for $t$-intersecting families of permutations, {\em J. Comb. Theory, Ser. A} {\bf 118} (2011) 208--227.

\bibitem{Ellis2} D. Ellis, E. Friedgut, and H. Pilpel, Intersecting families of permutations, {\em J. Am. Soc.} {\bf 24} (2011) 649--682.



\bibitem{Erdos} P. Erd{\H o}s, A problem on independent r-tuples, {\em Ann. Univ. Sci. Budapest} {\bf 8} (1965) 93--95.


\bibitem{EG} P. Erd{\H o}s, and T. Gallai, On maximal paths and circuits of graphs, {\em Acta Math. Acad. Sci. Hung.} {\bf 10} (1959) 337--356.

\bibitem{EKR} P. Erd{\H o}s, C. Ko, and R. Rado, Intersection
theorems for systems of finite sets, {\em Quart. J. Math. Oxford
Ser.} 2 {\bf 12} (1961), 313--318.


\bibitem{Frankl1} P. Frankl, Improved bounds for Erd{\H o}s’ Matching Conjecture, {\em J. Comb. Theory, Ser. A} {\bf 120} (2013) 1068--1072.

\bibitem{Frankl2} P. Frankl, Proof of the Erd{\H o}s Matching Conjecture in a new range, {\em Isr. J. Math.} {\bf 222}  (2017) 421--430.


\bibitem{Frankl3} P. Frankl, On the maximum number of edges in a hypergraph with a given matching number, {\em Discrete Appl. Math.} {\bf 216} (2017) 562--581.



\bibitem{FK2} P. Frankl, and A. Kupavskii, The Erd{\H o}s Matching Conjecture and concentration inequalities, {\em J. Combin. Theory Ser. B} {\bf 157} (2022), 366--400.


\bibitem{FRR} P. Frankl, V. R\"odl, and A. Ruci\'nski, On the maximum number of edges in a triple system not containing a disjoint family of a given size, {\em Comb. Probab. Comput.} {\bf 21} (2012) 141--148.

\bibitem{Huang} H. Huang, P. S. Loh, and B. Sudakov, The size of a hypergraph and its matching number, {\em Comb. Probab. Comput.} {\bf 21}  (2012) 442--450.


\bibitem{Ku} C.Y. Ku and D. Renshaw, Erd{\H o}s-Ko-Rado theorems for permutations and set partitions, {\em J. Comb. Theory, Ser. A} {\bf 115} (2008) 1008--1020.


\bibitem{Ku_Wong4} C. Y. Ku and K. B. Wong, An analogue of the Erd{\H o}s-Ko-Rado theorem for weak compositions, {\em Discrete Math.} {\bf 313} (2013), 2463-2468.

\bibitem{Ku_Wong41} C.Y. Ku, and K.B. Wong, An Erd{\H o}s-Ko-Rado theorem for permutations with fixed number of cycles, {\em Electron. J. Comb.} {\bf 21}(3) (2014) 3.16.

\bibitem{Ku_Wong5} C. Y. Ku and K. B. Wong, On $r$-cross $t$-intersecting families for weak compositions, {\em Discrete Math.} {\bf 338} (2015),  1090--1095.

\bibitem{Ku_Wong6} C. Y. Ku and K. B. Wong, An Analogue of the Hilton-Milner Theorem  for weak compositions, {\em Bull. Korean Math. Soc.} {\bf 52} (2015), 1007--1025.

\bibitem{Ku_Wong61} C.Y. Ku, and K.B. Wong, A non-trivial intersection theorem for permutations with fixed number of cycles, {\em Discrete Math.} {\bf 339} (2016) 646--657.


\bibitem{Ku_Wong7} C. Y. Ku and K. B. Wong, Diversity and intersecting theorems for weak compositions, preprint.


\bibitem{Ku_Wong8} C. Y. Ku and K. B. Wong, An analogue of the Erd{\H o}s Matching Conjecture for weak compositions, preprint.


\bibitem{Luc} T. \L{}uczak, and K. Mieczkowska, On Erd{\H o}s’ extremal problem on matchings in hypergraphs, {\em J. Comb. Theory, Ser. A} {\bf 124} (2014) 178--194.



\end{thebibliography}

}
\end{document}